\documentclass[review]{elsarticle}
\usepackage{amsthm}
\usepackage{lineno,hyperref}
\usepackage{amssymb,enumerate}
\usepackage{amsmath}
\usepackage{graphics}
\newtheorem{thm}{Theorem}[section]
 
 \newtheorem{lem}[thm]{Lemma}
 \newtheorem{prop}[thm]{Proposition}
 \newtheorem{defn}[thm]{Definition}
 \newtheorem{rem}[thm]{Remark}
 
\newcommand{\f}{\frac}

\newcommand{\Real}{{\bf R}}
\newcommand{\eps}{\varepsilon}

\def\R{{\bf R}}

\def\N{{\bf N}}

\def\d{\displaystyle}
\def\e{{\varepsilon}}

\def\vp{\varphi}

\journal{Journal of \LaTeX\ Templates}

\begin{document}

\begin{frontmatter}

\title{Lifespan of semilinear generalized Tricomi equation with Strauss type exponent}

\author{Jiayun Lin
\footnote{
Department of Mathematics and Science, School of Sciences, Zhejiang Sci-Tech University, 310018, Hangzhou, P.R.China.
 email: jylin@zstu.edu.cn.}\quad
Ziheng Tu
\footnote{School of Data Science, Zhejiang University of Finance and Economics, 310018, Hangzhou, P.R.China.
 e-mail: tuziheng@zufe.edu.cn.}
}

\begin{abstract}
In this paper, we consider the blow-up problem of semilinear generalized Tricomi equation. Two blow-up results with lifespan upper bound are obtained under subcritical and critical Strauss type exponent. In the subcritical case, the proof is based on the test function method and the iteration argument. In the critical case, an iteration procedure with the slicing method is employed. This approach has been successfully applied to the critical case of semilinear wave equation with perturbed Laplacian or the damped wave equation of scattering damping case. The present work gives its application to the generalized Tricomi equation.
\end{abstract}


\end{frontmatter}

\section{Introduction and Main result}
In this paper, we consider following initial value problem of semilinear generalized Tricomi equation
\begin{equation}\label{main}
\left\{
\begin{array}{ll}
u_{tt}-t^m\Delta u=|u|^p\ &(x,t)\ \in\ \mathbb{R}^n\times[0,\infty),\\
u(x,0)=\varepsilon u_0(x),\ u_t(x,0)=\varepsilon u_1(x)\ &x\ \in\ \mathbb{R}^n,\\
\end{array}
\right.
\end{equation}
where $m>0$, $u_0,\ u_1\in C_0^{\infty}(\mathbb{R}^n)$ and $n\in N$. Throughout this paper, we assume that $\varepsilon>0$ is a small parameter and the initial data $(u_0,\ u_1)$ has compact support
\begin{equation}
\mbox{supp}\ (u_0,\ u_1)\subset\left\{x\left|\ |x|\leqslant R\right.\right\}.
\end{equation}
We concern on the blow-up problem when the nonlinearity is critical. In addition, the lifespan upper bound for both critical and sub-critical case will be provided.

Before introducing the main results, we recall some related results from the historical point of view. When $m=0$, the Cauchy problem (\ref{main}) would reduce to semilinear wave equation:
\begin{equation}\label{wave}
\left\{
\begin{array}{ll}
u_{tt}-\Delta u=|u|^p\ &(x,t)\ \in\ \mathbb{R}^n\times[0,\infty),\\
u(x,0)=\varepsilon u_0(x),\ u_t(x,0)=\varepsilon u_1(x)\ &x\ \in\ \mathbb{R}^n.\\
\end{array}
\right.
\end{equation}
It is well-known that W. Strauss \cite{Strauss} made the conjecture that for \eqref{wave} there exists the critical exponent $p_S(n):=\frac{n+1+\sqrt{n^2+10n-7}}{2(n-1)},$ which is the positive root of the quadratic equation:
\begin{equation}\label{qua}
2+(n+1)p-(n-1)p^2=0,
\end{equation}
in the sense that: there exits global solution for small data if $p>p_S(n)$, and the solution would blow up in finite time even for the small data if $1<p\leqslant p_S(n)$. The Strauss conjecture was proved by F. John \cite{John} in the case $n=3$ and R.T. Glassey \cite{G1,G2} in the case $n=2$. When $n\geqslant 4$, T.C. Sideris \cite{Sideris} showed the blow-up of the solution in the case $1<p<p_S(n)$ and V. Georgiev, H. Lindblad, C.D. Sogge \cite{GLS} obtained the global solution for the small data in the case $p>p_S(n)$. In the critical case $p=p_S(n)$, J. Schaeffer \cite{Schaeffer} showed that small data solution blow up in finite time in dimension $n=2,3$, and the same result was proved by B. Yordanov and Q.S. Zhang \cite{Yor} for dimension $n \geqslant 4$, which completed the proof of Strauss conjecture. Recently, K. Wakasa and B. Yordanov \cite{WY1801} revisit this problem but with metric perturbations of the Laplacian. They have used the John's iteration method in \cite{John} with the "slicing method" from R. Agemi, Y. Kurokawa and H. Takamura \cite{AKT00}.
Moreover, these methods are also applied to the damping wave equation with scattering case in \cite{WY1802}.

On the other hand, the Tricomi equation has also been studied extensively. For the linear equation, J. Barros-Neto and I.M. Gelfand in \cite{Bar} and K. Yagdjian \cite{Yag04} computed the fundamental solution explicitly. In \cite{Yag06}, K. Yagdjian investigated the issue of global existence of this semilinear problem. Some necessity conditions of nonlinearity exponent were derived for both blow-up and global existence. However, the critical exponent was not determined. Later, Z. Ruan, I. Witt and H. Yin \cite{Rua1,Rua2,Rua3,Rua4} gave a systematically studied of the semilinear problem. The local existence result and minimal regularity solution are established. With expectation that the semilinear generalized Tricomi equation shall share similar properties as semilinear wave equation,  D. He, I. Witt and H. Yin have found the critical exponent $p_{crit}(m,n)$ as the the positive root of following quadratic equation
\begin{equation}
\gamma(m,n,p):=-((m+2)\f n2-1)p^2-((m+2)(1-\f n2)-3)p+(m+2)=0.
\end{equation} They have shown that the solution to problem \eqref{main} would blow up in finite time in the sub-critical case $1<p<p_{crit}(m,n)$ in \cite{He-Witt-Yin17}, and the small data solution would exist globally in the super-critical case $p>p_{crit}(m,n)$ in \cite{He-Witt-Yin II}.
The results of low dimension case $n=1,\ 2$ were given in \cite{He-Witt-Yin1D,He-Witt-Yin2D}. In particular, they also showed the blow-up result for the critical case for $m=1$ in \cite{He-Witt-Yin2D}, by applying the test function method and the Riccati-type ordinary differential inequality. In this sense, they have determined the critical exponent of semilinear Tricomi equation. However, the blow-up result of critical part is only given for $m=1$ and the lifespan estimates of subcritical and critical case are still left unknown.

In present paper, our main purpose is to consider the blow-up problem of critical case for general $m$. By applying the approach as K. Wakasa and B. Yordanov in \cite{WY1801}, we obtained the expected blow-up result. The core part is to find the fundamental solution system of \eqref{Tricomi-ODE} and give the asymptotic estimation. In fact, by variable changing, the solution can be written simply by the confluent hypergeometric function. The test function with some element estimation is then derived with the aid of these functions. Moreover, by employing the iteration argument, the lifespan upperbound estimates for both sub-critical case and critical case are established. We emphasis that this approach can be easily extend to the perturbed Laplacian.

We now introduce the definition of weak solution and the main result.
\begin{defn}
Define $u$ is an energy space solution, if
$(u,u_t)\in C([0,T_\varepsilon), H^1(\Real^n)\times L^2(\Real^n))$ and for any $\zeta\in C_0^\infty(\Real^n\times [0,T_\varepsilon))$
\begin{eqnarray}
\nonumber
& & \int u_s(x,t)\zeta(x,t) dx-\int u_s(x,0)\zeta(x,0) dx\\
\label{43}
& & -\int_0^t \int (u_s(x,s)\zeta_{s}(x,s)- u_1(x)s^m\nabla u(x,s)\cdot\nabla \zeta(x,s)) dxds \label{weak}\\
\nonumber
& & =\int_0^t \int |u(x,s)|^p \zeta(x,s) dxds,
\end{eqnarray}
for $t\in (0,T_\eps).$
\end{defn}

Our main results are stated in the following.
\begin{thm}\label{Thm:sub-critical}
For the Cauchy problem \eqref{main} with $1<p<p_{crit}(m,n)$, let the initial values $u_0,\ u_1$ be nonnegative smooth function with compact support in $\{x||x|\leqslant R\}$. Suppose that a solution $u$ of \eqref{main} satisfies:
$$supp\ u\subset\{(x,t)\in \mathbb{R}^n\times [0,T): |x|\leqslant B(0,R+\phi(t))\},$$
where $\phi(t)=\f2{m+2}t^{\f{m+2}2}$. Then the lifespan $T<\infty$ and there exists a positive constant $C$ which is independent of $\varepsilon$ such that
$$T(\varepsilon)\leqslant C\varepsilon^{-\frac{2p(p-1)}{\gamma(m,n,p)}}.$$
\end{thm}

\begin{thm}\label{Thm:critical}
For the Cauchy problem \eqref{main} with $p=p_{crit}(m,n)$, let the initial values $u_0,\ u_1$ be nonnegative smooth function with compact support in $\{x||x|\leqslant R\}$. Suppose that a solution $u$ of \eqref{main} satisfies:
$$supp\ u\subset\{(x,t)\in \mathbb{R}^n\times [0,T): |x|\leqslant B(0,R+\phi(t))\},$$
where $\phi(t)=\f2{m+2}t^{\f{m+2}2}$. Then the lifespan $T<\infty$ and there exists a positive constant $C$ which is independent of $\varepsilon$ such that
$$T(\varepsilon)\leqslant \exp(C\varepsilon^{-p(p-1)}).$$
\end{thm}

We arrange our proof as follows. In Section 2, we study the fundamental solutions of generalized Tricomi equation related ODE \eqref{Tricomi-ODE} to give some primaries proof elements including the test function and its asymptotic behavior. Section 3 concerns on the logarithmic type integral inequality which plays a key role in proceeding the iteration argument for critical case. Additionally, the lower bound of $L^p$ norm of solution is derived from this logarithmic type integral inequality. In Section 4, we focus on the proof of main theorems. The iteration argument are employed for both sub-critical case and critical case.

\section{Test functions}
 We first consider the "prototype" of the generalized Tricomi equation,
\begin{equation}\label{Tricomi-ODE}
y''-\lambda^2t^my=0.
\end{equation}
Proceeding as Section 3.1 in \cite{Yag06}, we may obtain its fundamental solution through the confluent hypergeometric function. To be precisely,
let $z=-\f{4\lambda}{m+2}t^{\f{m+2}2}$ and we introduce the unknown function $W(z)$ such that
\begin{equation}\label{W}y(t)=W(z)\exp(-\f{z}2).\end{equation}
Plugging the derivatives
\begin{eqnarray*}y'(t)&=&[W'-\f12W]\f{dz}{dt}\exp(-\f z2)\\
y''(t)&=&[W''-\f12W'](\f{dz}{dt})^2\exp(-\f z2)\\
&+&[W'-\f12W](\f{d^2z}{dt^2})\exp(-\f z2)\\
&+&[W'-\f12W](\f{dz}{dt})^2(-\f12)\exp(-\f z2),
\end{eqnarray*}
into ODE \eqref{Tricomi-ODE}, we find that $W(z)$ satisfies the confluent hypergeometric equation (or Kummer's equation)
$$zW''(z)+(\f{m}{m+2}-z)W'-\f{m}{2(m+2)}W=0.$$
There exits two linearly independent solution $M(\alpha,\gamma;z)$ and $z^{1-\gamma}M(1+\alpha-\gamma,2-\gamma;z)$ for this equation, with
$$\alpha=\f{m}{2(m+2)},\ \ \gamma=\f{m}{m+2}.$$
The function $M(\alpha,\gamma;z)$ is called confluent hypergeometric function, which is entire in $z$. We collect some properties for our later use. Interesting reader may refer \cite{Erdelyi} for more details.
\begin{itemize}
  \item Limiting form as $z\rightarrow0$: $$M(a,b;z)=1+O(z).$$
  \item Limiting form as $z\rightarrow\infty$:
  \begin{equation}\label{M-asy1}M(a,b;z)\sim \f{\Gamma(b)}{\Gamma(a)}e^zz^{a-b},\ \ \ Re\ z>0. \end{equation}
  \begin{equation}\label{M-asy2}M(a,b;z)\sim \f{\Gamma(b)}{\Gamma(b-a)}(-z)^{-a},\ \ \ Re\ z<0. \end{equation}
  \item Kummer's transformation:
  $$M(a,b;z) = e^zM(b-a,b;-z).$$
  \item Derivative and Wronskian:
  $$\f{d}{dz}M(a,b;z)=\f abM(a+1,b+1;z).$$
  \begin{eqnarray*}
  \mathcal{W}(M(a,b;z),z^{1-b}M(1+a-b,2-b;z))  
  =(1-b)z^{-b}e^z.
  \end{eqnarray*}
  \item Relations to elementary functions:
  \begin{equation}\label{ele}M(0,0;z)=e^z,\ M(1,2;z)=\f{2e^{\f z2}}z\sinh\f z2.\end{equation}
\end{itemize}

Now, we return to the ODE \eqref{Tricomi-ODE}. With the help of the independent solutions of Kummer's equation, the following lemma gives its fundamental solution system through \eqref{W}.
\begin{lem}Let $z=-2\lambda\phi(t)$, $\phi(t)=\f{2}{m+2}t^{\f{m+2}2}$, $\alpha=\f{m}{2(m+2)}$, $\gamma=\f{m}{m+2}$, then the functions
$$V_1(t;\lambda)=e^{-\f z2}M(\alpha,\gamma;z),\ \ \ V_2(t;\lambda)=e^{-\f z2}c_mz^{1-\gamma}M(1+\alpha-\gamma,2-\gamma;z)$$
form the fundamental system for the equation \eqref{Tricomi-ODE} such that
$$V_1(0;\lambda)=1,\ V_1'(0;\lambda)=0;\ \ V_2(0;\lambda)=0,\ V_2'(0;\lambda)=1.$$
$c_m=(-\f{4\lambda}{m+2})^{-\f2{m+2}}$ is a scaling coefficient such that $V_2'(0;\lambda)=1$.
Moreover, for $t\geqslant s\geqslant0$, denote the special solutions $\Phi_1(t,s;\lambda)$ and $\Phi_2(t,s;\lambda)$ by the determinant
\begin{equation}
\Phi_1(t,s;\lambda)=\left|\begin{array}{cc}V_1(t;\lambda)&V_2(t;\lambda)\\\partial_tV_1(s;\lambda)&\partial_tV_2(s;\lambda)\end{array}\right|
\end{equation}
\begin{equation}
\Phi_2(t,s;\lambda)=\left|\begin{array}{cc}V_1(s;\lambda)&V_2(s;\lambda)\\V_1(t;\lambda)&V_2(t;\lambda)\end{array}\right|
\end{equation} we then have:
$$\Phi_1(s,s;\lambda)=1,\ \partial_t\Phi_1(s,s;\lambda)=0;\ \ \Phi_2(s,s;\lambda)=0,\ \partial_t\Phi_2(s,s;\lambda)=1.$$
\end{lem}
\begin{proof}It is easy to verify $V_1(t;\lambda)$ and $V_2(t;\lambda)$ are the required fundamental solution. We only give the verification for $\Phi_1(t,s;\lambda)$ and $\Phi_2(t,s;\lambda)$. First, it is obvious that
$$\partial_t\Phi_1(s,s;\lambda)=\Phi_2(s,s;\lambda)=0.$$
Second, by using the Wronskian of $M(\alpha,\gamma;z)$ and $z^{1-\gamma}M(1+\alpha-\gamma,2-\gamma;z)$, we have
\begin{eqnarray*}
&&\Phi_1(s,s;\lambda)=\partial_t\Phi_2(s,s;\lambda)=\left|\begin{array}{cc}V_1(s;\lambda)&V_2(s;\lambda)\\ \partial_t V_1(s;\lambda)&\partial_tV_2(s;\lambda)\end{array}\right|\\
&=&\left|\begin{array}{cc}e^{-\f z2}M(\alpha,\gamma;z)&e^{-\f z2}c_mz^{1-\gamma}M(1+\alpha-\gamma,2-\gamma;z)\\\partial_t\left[e^{-\f z2}M(\alpha,\gamma;z)\right]&\partial_t\left[e^{-\f z2}c_mz^{1-\gamma}M(1+\alpha-\gamma,2-\gamma;z)\right]\end{array}\right|\\
&=&\left|\begin{array}{cc}e^{-\f z2}M(\alpha,\gamma;z)&e^{-\f z2}c_mz^{1-\gamma}M(1+\alpha-\gamma,2-\gamma;z)\\e^{- \f z2}\partial_t M(\alpha,\gamma;z)&e^{-\f z2}c_m\partial_t\left(z^{1-\gamma}M(1+\alpha-\gamma,2-\gamma;z)\right)\end{array}\right|\\
&=&e^{-z}c_m\f{dz}{dt}\left|\begin{array}{cc}M(\alpha,\gamma;z)&z^{1-\gamma}M(1+\alpha-\gamma,2-\gamma;z)\\\partial_z M(\alpha,\gamma;z)&\partial_z\left(z^{1-\gamma}M(1+\alpha-\gamma,2-\gamma;z)\right)\end{array}\right|\\
&=&c_me^{-z}(-2\lambda t^{\f{m}2})\times\mathcal{W}(M(\alpha,\gamma;z),z^{1-\gamma}M(1+\alpha-\gamma,2-\gamma;z))\\
&=&c_m e^{-z}(-2\lambda t^{\f{m}2})\times (1-\gamma)z^{-\gamma}e^z=1.
\end{eqnarray*}\end{proof}
\begin{rem}
It is clearly, $c_mz^{1-\gamma}=t$ and
$$V_1(t;\lambda)=\Phi_1(t,0;\lambda),\ \ V_2(t;\lambda)=\Phi_2(t,0;\lambda).$$
\end{rem}
Next, we define following two test function. For $\lambda_0\in (0,\beta/2]$ and $q>-1$, let
\begin{eqnarray}
\label{aq}
\xi_q (x,t,s) & = & \int_{0}^{\lambda_0}e^{-\lambda(\phi(t)+R)}\Phi_1(t,s;\lambda) \: \varphi_\lambda(x)\lambda^{q} d\lambda,\\
\eta_q (x,t,s) & = & \int_{0}^{\lambda_0}e^{-\lambda(\phi(t)+R)}\frac{\Phi_2(t,s;\lambda)}{t-s}\: \varphi_\lambda(x)\lambda^{q} d\lambda.
\label{bq}
\end{eqnarray}
Here $\varphi_\lambda(x)=\varphi(\lambda x)$ with
\begin{equation*}\varphi(x):=
\begin{cases}
\int_{\mathbb{S}^{n-1}}e^{x\cdot \omega}d\omega \ \ &\mbox{for}\ \ n\geqslant2,\\
e^x+e^{-x} \ \ &\mbox{for}\ \ n=1.
\end{cases}
\end{equation*}
The function $\varphi$  satisfies $$\Delta\varphi(x)=\varphi(x)$$ and the asymptotic estimate
\begin{equation}\label{asy-varphi}
\varphi(x)\sim C_n|x|^{-\frac{n-1}2}e^{|x|} \ \ \ \mbox{as}\ \ \ |x|\rightarrow\infty.
\end{equation}
Besides, by L'Hospital's rule, one may also verify that
$$\lim_{s\rightarrow t}\frac{\Phi_2(t,s;\lambda)}{t-s}=1,$$
which implies that
$$\eta_q (x,t,t) = \int_{0}^{\lambda_0}e^{-\lambda(\phi(t)+R)}\: \varphi_\lambda(x)\lambda^{q} d\lambda.$$
We collect some element estimates about $\xi_q(x,t)$ and $\eta_q(x,t,s)$ in following lemma.
\begin{lem}
\label{lem2}
Let $n\geqslant 2$. There exists $\lambda_0\in (0,\beta/2]$, such that the following hold:

(i) if $-\f{m}{2(m+2)}<q$, $|x|\leqslant R$ and $0\leqslant t$, then
\begin{eqnarray*}
\xi_q (x,t,0) & \geqslant & A_0\langle\phi(t)\rangle^{-\f m{2(m+2)}},\\
\eta_q (x,t,0) & \geqslant & B_0\langle\phi(t)\rangle^{-\f{m+4}{2(m+2)}};
\end{eqnarray*}

(ii) if $-\f{m}{2(m+2)}<q$, $|x|\leqslant \phi(s)+R$ and $0\leqslant s<t$, then
\begin{eqnarray*}
\eta_q (x,t,s) & \geqslant & B_1 \langle t\rangle^{-1-\f{m}4}\langle \phi(s)\rangle^{-q-1+\f{m+4}{2(m+2)}};
\end{eqnarray*}

(iii) if $(n-3)/2<q$, $|x|\leqslant \phi(t)+R$ and $0<t$, then
\begin{eqnarray*}
\eta_q (x,t,t) & \leqslant & B_2 \langle \phi(t)\rangle^{-(n-1)/2}\langle \phi(t)-|x| \rangle^{(n-3)/2-q}.
\end{eqnarray*}
Here $A_0$ and $B_k$, $k=0,1,2,$ are positive constants depending only on $\beta$, $q$ and $R$,
while $\langle s\rangle =3+|s|$ is used to simplify estimates in Sections 4 and 5.
\end{lem}
\begin{proof} $(i)$ Since $z=-2\lambda\phi(t)<0$, we shall apply \eqref{M-asy2} for $M(\alpha,\gamma;z)$ when $|z|$ is large. For small $|z|$, we only need to require the integral is convergent around $0$, i.e, $q-\alpha+1>0$ and $q+\alpha>0$. By \eqref{aq},
\begin{eqnarray*}
\xi_q(x,t,0) & \geqslant & \inf_{0\leqslant \lambda\leqslant \lambda_0}\inf_{|x|\leqslant R}\vp_\lambda(x)  \int_{0}^{\lambda_0}e^{-\lambda R}M(\alpha,\gamma;z)\lambda^qd\lambda\\
 & \geqslant & \inf_{0\leqslant \lambda\leqslant \lambda_0}\inf_{|x|\leqslant R}\vp_\lambda(x) \int_{\f {\lambda_0}2}^{\lambda_0}e^{-\lambda R}\f{\Gamma(\gamma)}{\Gamma(\gamma-\alpha)}(-z)^{-\alpha}\lambda^qd\lambda\\
 &\geqslant& A_0 \langle\phi(t)\rangle^{-\alpha}
\end{eqnarray*}
Similarly, by (\ref{bq}) we have,
\begin{eqnarray*}
\eta_q (x,t,0) & \geqslant & \inf_{0\leqslant \lambda\leqslant \lambda_0}\inf_{|x|\leqslant R}
\vp_\lambda(x)\int_{0}^{\lambda_0}e^{-\lambda R}
M(1+\alpha-\gamma,2-\gamma;z)\lambda^{q}\: d\lambda\\
&\geqslant& \inf_{0\leqslant \lambda\leqslant \lambda_0}\inf_{|x|\leqslant R}\vp_\lambda(x) \int_{\f {\lambda_0}2}^{\lambda_0}e^{-\lambda R}\f{\Gamma(2-\gamma)}{\Gamma(1-\alpha)}(-z)^{-1-\alpha+\gamma}\lambda^qd\lambda\\
& \geqslant& B_0\langle\phi(t)\rangle^{-1-\alpha+\gamma}.
\end{eqnarray*}
$(ii)$ We combine (\ref{bq}) and the positivity of $\vp_\lambda(x)$ from \eqref{asy-varphi}. Then
\begin{eqnarray*}
|\eta_q (x,t,s) |& = & \int_{0}^{\lambda_0}e^{-\lambda(\phi(t)-\phi(s))}\f{|\Phi_2(t,s;\lambda)|}{t-s}\:
[e^{-\lambda(\phi(s)+R)}\vp_\lambda(x)]\lambda^q d\lambda\\
&\geqslant&A_1\int_{\lambda_0/\langle\phi(s)\rangle}^{2\lambda_0/\langle\phi(s)\rangle}e^{-\lambda(\phi(t)-\phi(s))}\f{|\Phi_2(t,s;\lambda)|}{t-s}\:\lambda^q d\lambda
\end{eqnarray*}
with $$A_1=\inf_{\lambda_0/\langle\phi(s)\rangle\lambda\leqslant2\lambda_0/\langle\phi(s)\rangle}\inf_{|x|<\phi(s)+R}e^{-\lambda(\phi(s)+R)}\varphi_\lambda(x)>0,$$
which is independent with $s$ and $x$.
By using the Kummer's transformation, we have
 \begin{eqnarray*}
\Phi_2(t,s;\lambda)&=&\left|\begin{array}{cc}e^{-\f{z(s)}2}M(\alpha,\gamma;z(s))&e^{-\f{z(s)}2}c_mz(s)^{1-\gamma}M(1+\alpha-\gamma,2-\gamma;z(s)) \\e^{-\f{z(t)}2}M(\alpha,\gamma;z(t))&e^{-\f{z(t)}2}c_mz(t)^{1-\gamma}M(1+\alpha-\gamma,2-\gamma;z(t)) \end{array}\right|\\
&=&\left|\begin{array}{cc}e^{\f{z(s)}2}M(\gamma-\alpha,\gamma;-z(s))& e^{\f{z(s)}{2}}c_mz(s)^{1-\gamma}M(1-\alpha,2-\gamma;-z(s))\\e^{-\f{z(t)}2}M(\alpha,\gamma;z(t))&e^{-\f{z(t)}2}c_mz(t)^{1-\gamma}M(1+\alpha-\gamma,2-\gamma;z(t))
\end{array}\right|\\
&=&e^{(z(s)-z(t))/2}\left[M(\gamma-\alpha,\gamma;-z(s))tM(1+\alpha-\gamma,2-\gamma;z(t))\right.\\
&&\left.-M(\alpha,\gamma;z(t))sM(1-\alpha,2-\gamma;-z(s))\right].
\end{eqnarray*}
Hence,
\begin{eqnarray*}
|\eta_q (x,t,s) |&\geqslant & A_1\int_{\lambda_0/\langle \phi(s)\rangle }^{2\lambda_0/\langle \phi(s)\rangle }\frac{1}{t-s}\left|M(\gamma-\alpha,\gamma;-z(s))tM(1+\alpha-\gamma,2-\gamma;z(t))\right.\\
&&\left.-M(\alpha,\gamma;z(t))sM(1-\alpha,2-\gamma;-z(s))\right|\:\lambda^{q} d\lambda,
\end{eqnarray*}
As $\lambda_0\leqslant\lambda\langle \phi(s)\rangle\leqslant2\lambda_0$, we have $-z(s)=2\lambda\phi(s)$ is bounded in $(\lambda_0,2\lambda_0)$, so
$M(\gamma-\alpha,\gamma;-z(s))$ and $M(1-\alpha,2-\gamma;-z(s))$ is bounded which is independent with $\lambda$ and $s$. On the other hand, for large $t$
$$M(1+\alpha-\gamma,2-\gamma;z(t))\sim (-z(t))^{-1-\alpha+\gamma}, $$
$$M(\alpha,\gamma;z(t))\sim (-z(t))^{-\alpha}.  $$
We may simplify the lower estimate of $|\eta_q (x,t,s)|$ as
\begin{eqnarray*}
|\eta_q (x,t,s) |
&\geqslant &\frac{{A_1}}{t-s}\int_{\lambda_0/\langle \phi(s)\rangle }^{2\lambda_0/\langle \phi(s)\rangle }\bigg|C_{M_1}t(-z(t))^{-1-\alpha+\gamma}-C_{M_2}s(-z(t))^{-\alpha}\bigg|\:\lambda^{q} d\lambda\\
&\geqslant &A_1\frac{t^{-\f m4}}{t-s}\int_{\lambda_0/\langle \phi(s)\rangle }^{2\lambda_0/\langle \phi(s)\rangle }\bigg|C_{M_1}\lambda^{-\f{m+4}{2(m+2)}}-C_{M_2}s\lambda^{-\f{m}{2(m+2)}}\bigg|\:\lambda^{q} d\lambda,
\end{eqnarray*}
where the constants $C_{M_1}$ and $C_{M_2}$ are independent with $\lambda$, $s$ and $t$.
Since we assume that $\langle s\rangle\geqslant 2$ for all $s\in\R$ and $\lambda\sim \f{\lambda_0}{\langle \phi(s)\rangle }$, we find the two terms in the absolute value brackets are actually in the same order of $\lambda$. Moreover, we can always take small $\lambda_0$, so that $|C_{M_1}-C_{M_2}\lambda_0^{\f2{m+2}}|\geqslant\frac12C_{M_1}$. Finally we obtain for $q>-\f{m}{2(m+2)}$,
\begin{eqnarray*}|\eta_q (x,t,s) |&\geqslant& A_1\frac{t^{-\f m4}}{t-s}\int_{\lambda_0/\langle \phi(s)\rangle }^{2\lambda_0/\langle \phi(s)\rangle }\bigg|C_{M_1}-C_{M_2}\lambda_0^{\f2{m+2}}\bigg|\lambda^{-\f{m+4}{2(m+2)}}\:\lambda^{q} d\lambda\\
&\geqslant& \widetilde{ A_1}\frac{t^{-\f m4}}{t-s}\int_{\lambda_0/\langle \phi(s)\rangle }^{2\lambda_0/\langle \phi(s)\rangle }\lambda^{-\f{m+4}{2(m+2)}}\:\lambda^{q} d\lambda\\
&\geqslant& \widetilde{ A_1}\frac{t^{-\f m4}}{t-s}\langle \phi(s)\rangle^{-q-1+\f{m+4}{2(m+2)}}\geqslant B_1 \langle t\rangle^{-1-\f {m}4}\langle \phi(s)\rangle^{-q-1+\f{m+4}{2(m+2)}}.
\end{eqnarray*}
$(iii)$ Substituted \eqref{asy-varphi} into (\ref{bq}) to derive:
\begin{eqnarray*}
\eta_q (x,t,t) & \leqslant & D_0^{-1} \int_0^{\lambda_0} \frac{e^{-\lambda(\phi(t)+R-|x|)}\lambda^{q} }{\langle\lambda|x|\rangle^{(n-1)/2}}\: d\lambda.
\end{eqnarray*}
It is convenient to consider two cases. If $|x|\leqslant (\phi(t)+R)/2$, the estimate becomes
$$
\eta_q (x,t,t) \leqslant  D_1 \int_0^{\lambda_0} e^{-\lambda(\phi(t)+R)/2}\lambda^{q} d\lambda\leqslant D_2 \langle \phi(t)\rangle^{-q-1}.
$$
If $|x|\geqslant (\phi(t)+R)/2$, the resulting bound is different:
\begin{eqnarray*}
\eta_q (x,t,t) & \leqslant  & D_0^{-1} \langle |x|\rangle^{-(n-1)/2}\int_0^{\lambda_0} e^{-\lambda(\phi(t)+R-|x|)}\lambda^{q-(n-1)/2}
                      \: d\lambda\\
            & \leqslant &  D_3\langle |x| \rangle^{-(n-1)/2}\langle \phi(t)-|x|\rangle^{(n-3)/2-q}.
\end{eqnarray*}
Clearly, both results are included into $\eta_q (x,t,t) \leqslant B_2 \langle \phi(t) \rangle^{-(n-1)/2}\langle \phi(t)-|x|\rangle^{(n-3)/2-q}.$
\end{proof}
\begin{rem}
It is easy to verify the above arguments and results are general extension of wave equation by simply putting $m=0$, if one notice that \eqref{ele}.
\end{rem}

\section{Nonlinear integral inequality}
In this section, based on the estimation of test function in previous section, we give the lower bound of weighted functional
\begin{equation}
\label{def:F}
F(t)=\int_{\mathbb{R}^n} u(x,t) \eta_{q}(x,t,t)dx.
\end{equation}
By testing equation \eqref{main} by $\vp_\lambda(x)$, we find that
$$\frac{d^2}{dt^2}\int_{\mathbb{R}^n}u(x,t)\vp_\lambda(x)dx-t^m\lambda^2\int_{\mathbb{R}^n}u(x,t)\vp_\lambda(x)dx=\int_{\mathbb{R}^n}|u|^p\vp_\lambda(x)dx.$$
That is
$$G''(t)-t^m\lambda^2G(t)=\int_{\mathbb{R}^n}|u|^p\vp_\lambda(x)dx,$$
if we define the functionals
$$G(t;\lambda)=\int_{\mathbb{R}^n} u(x,t) \vp_\lambda(x)dx.$$
Applying the Duhamel's principle, we may solve $G(t;\lambda)$ by following integral representation
\begin{eqnarray*}
G(t;\lambda)=\int_{\mathbb{R}^n}u(x,t)\vp_\lambda(x)dx&=&\varepsilon\Phi_1(t,0;\lambda)\int_{\mathbb{R}^n}u_0(x)\vp_\lambda(x)dx\\
&+&\varepsilon\Phi_2(t,0;\lambda)\int_{\mathbb{R}^n}u_1(x)\vp_\lambda(x)dx\\
&+&\int_0^t\Phi_2(t,s;\lambda)\left(\int_{\mathbb{R}^n}|u(s)|^p\vp_\lambda(x)dx\right)ds.
\end{eqnarray*}
Multiplying by $\lambda^q e^{-\lambda(\phi(t)+R)}$, integrating on $[0,\lambda_0]$ and interchanging the order of integration between $\lambda$ and $x$, we derive following identity
\begin{eqnarray}
F(t)=\int_{\mathbb{R}^n}u(x,t)\eta_q(x,t,t)dx&=&\varepsilon\int_{\mathbb{R}^n}u_0(x)\xi_q(x,t,0)dx\nonumber\\
&+&\varepsilon t\int_{\mathbb{R}^n}u_1(x)\eta_q(x,t,0)dx\label{identity}\\
&+&\int_0^t(t-s)\int_{\mathbb{R}^n}|u(s)|^p\eta_q(x,t,s)dxds.\nonumber
\end{eqnarray}
We now in the position to give the iteration argument frame for the critical case. In fact, we have following lower bound of logarithmic type within the nonlinear terms integral.
\begin{prop}
\label{prop:frame}
Suppose that the assumptions in Theorem \ref{Thm:critical} are fulfilled and choose
$q=(n-1)/2-1/p.$ If $F(t)$ is defined in (\ref{def:F}), there exists a positive constant $C=C(n,p,R)$, such that
\begin{equation}
\label{frame}
\langle t\rangle^{\f m4}F(t)  \geqslant \frac{C}{\langle t\rangle}
\int_0^t \frac{t-s}{\langle s\rangle}\frac{(\langle s\rangle^{\f m4}F(s))^p}{
(\log \langle s\rangle)^{p-1}}\: ds
\end{equation}
for all $t\in (0,T_\eps).$
\end{prop}

\begin{proof}
Let $0\leqslant s<t$. From $F(s)=\int_{\R^n} u(x,s) \eta_{q}(x,s,s)dx$ and H\"{o}lder's inequality,
\begin{equation}
\label{Holder:F1}
\begin{array}{lll}
\d |F(s)|\leqslant \left(\int |u(x,s)|^p \eta_{q}(x,t,s) dx \right)^{1/p} \\
\d\qquad \times \left( \int_{|x|\leqslant \phi(s)+R} \frac{\{\eta_{q}(x,s,s)\}^{p/(p-1)}}
{\{\eta_{q}(x,t,s)\}^{1/(p-1)}}
dx\right)^{(p-1)/p}.
\end{array}
\end{equation}
Substituting estimates $(ii)$ and $(iii)$ from Lemma~\ref{lem2} with $q=(n-1)/2-1/p$,
we can bound the second integral by
$$
C\int_{|x|\leqslant \phi(s)+R} \frac{\langle \phi(s)\rangle ^{-(n-1)p/2(p-1)} \langle \phi(s)-|x| \rangle ^{\{(n-3)/2-q\}p/(p-1)}}
{\langle t\rangle ^{(-\f m4-1)/(p-1)}\langle \phi(s) \rangle ^{-(q+1-\f{m+4}{2(m+2)})/(p-1)}}dx.
$$
This expression simplifies to
\begin{eqnarray*}
C\langle t\rangle ^{(1+\f m4)/(p-1)}\langle \phi(s) \rangle ^{[q+1-\f{m+4}{2(m+2)}-(n-1)p/2]/(p-1)}
\int_{|x|\leqslant \phi(s)+R} \langle \phi(s)-|x| \rangle ^{\{(n-3)/2-q\}p/(p-1)} dx.
\end{eqnarray*}
The latter integral is actually
$$
\int_{|x|\leqslant \phi(s)+R} \langle \phi(s)-|x| \rangle ^{-1} dx\leqslant C\langle \phi(s)\rangle^{n-1}\log \langle \phi(s)\rangle,
$$
so the final estimate of the second integral in (\ref{Holder:F1}) becomes
$$
C\langle t\rangle ^{(1+\f m4)/(p-1)}
 \langle \phi(s)\rangle ^{\f1p+\f{n-1}2-\f{m+4}{2(m+2)}(p-1)^{-1}}\log\langle s\rangle.
$$
From Lemma~\ref{lem2} and identity \eqref{identity}, we see that $F(t)\geqslant 0.$ Thus, (\ref{Holder:F1}) gives
$$
F(s)^p\leqslant C\langle t\rangle^{1+\f m4}  \langle \phi(s) \rangle^{1-\f1p+\f{n-1}2(p-1)-\f{m+4}{2(m+2)}} (\log\langle s\rangle)^{p-1}
\left(\int |u(x,t)|^p \eta_{q}(x,t,s) dx\right),
$$
i.e,
$$\int |u(x,t)|^p \eta_{q}(x,t,s) dx\geqslant\frac{ F(s)^p}{ \langle t\rangle^{1+\frac m4} \langle \phi(s) \rangle^{1-\f1p+\f{n-1}2(p-1)-\f{m+4}{2(m+2)}}
(\log\langle s\rangle)^{p-1}}.$$
Inserting this lower bound into \eqref{identity} and combining estimates $(i)$ in Lemma \ref{lem2}, we have that
\begin{eqnarray}
F(t)  &\geqslant&   C_1(u_0)\e \langle t\rangle^{-\f m4}+C_2(u_1)\e t\langle t\rangle^{-\f m4-1}\nonumber\\
 &+&\frac{C}{\langle t\rangle^{1+\f m4}}\int_0^t\frac{ (t-s)F(s)^p\: ds}{  \langle \phi(s) \rangle^{1-\f1p+\f{n-1}2(p-1)-\f{m+4}{2(m+2)}}
(\log\langle s\rangle)^{p-1}}.\label{frame-sub}
\end{eqnarray}
Since $p=p_s(m,n)$ implies that
$$\f{m}4p+[1-\f1p+\f{n-1}2(p-1)-\f{m+4}{2(m+2)}]\f{m+2}2=1,$$
one may reduce
$$F(t)\langle t\rangle^{\f m4}\geqslant \f{C}{\langle t\rangle}\int_0^t\f{(t-s)(F(s) \langle s\rangle^{\f m4})^p}{ \langle s \rangle(\log\langle s\rangle)^{p-1}}ds,$$
which completes the proof of this proposition.
\end{proof}

Lastly, as a preparation of iteration argument proceeded in next section, we show following lower bound of the $L^p$ norm of $u$.
The result of this lemma can be summarized from the proof of Section 2 in \cite{He-Witt-Yin17},
where the modified Bessel function is introduced as the test function. However, we here apply different approach by utilizing the estimation of test function $\eta_q(x,t,t)$ and the nonlinear integral inequality \eqref{frame-sub}.
\begin{lem}\label{lem:lp_lower}
Suppose that the assumptions in Theorem \ref{Thm:sub-critical} are fulfilled.
Then there exists large $T_0$ which is independent with $u_0,\ u_1$ and $\varepsilon$, for any $t>T_0$ and $p>1$,
\begin{equation}\label{lp_lower}
\int_{\mathbb{R}^n}|u(x,t)|^pdx\geqslant C_0\varepsilon^p(1+t)^{\f p2}(1+\phi(t))^{n-1-\frac{n}2 p}
\end{equation}
where the constant $C_0$ depends on $m,\ p$, $R$ and $T_0$.
\end{lem}
\begin{proof}
Making use of (\ref{frame-sub}) and H\"{o}lder's inequality, we get
\begin{equation}
\label{Holder_u^p}
C_1(u_0,u_1)\e \langle t\rangle^{-\f m4}\le |F(t)| \le \left(\int_{\R^n}|u(x,t)|^pdx\right)^{1/p}
\cdot (I(t))^{1/p'},
\end{equation}
where we set
$$
I(t)=\int_{|x|\le \phi(t)+R} \{\eta_{q}(x,t,t)\}^{p'}dx,
$$
and $p'=p/(p-1)$. If we have shown the upper bound of $I(t)$ satisfies
\begin{equation}\label{upp_I(t)}I(t)\le C\langle \phi(t)\rangle^{n-1-(n-1)p'/2},\end{equation}
Then \eqref{Holder_u^p} can be derived from \eqref{lp_lower} directly computation. Following we show \eqref{upp_I(t)}.
By using the estimates (iii) in Lemma \ref{lem2} with $q>(n-3)/2+1/p'$, one obtains
\begin{eqnarray*}
I(t)&\le& C\langle \phi(t)\rangle^{-(n-1)p'/2} \int_{|x|\le \phi(t)+R}\langle \phi(t)-|x|\rangle^{(n-3)p'/2-p'q}dx
\\ &=& C\langle \phi(t)\rangle^{-(n-1)p'/2}\int_{0}^{\phi(t)+R} r^{n-1} \langle \phi(t)-r\rangle^{(n-3)p'/2-p'q} dr
\end{eqnarray*}
Changing the variables by $\phi(t)-r=\rho$, we have
\begin{eqnarray*}
I(t)&\le& C\langle \phi(t)\rangle^{-(n-1)p'/2} \int_{-R}^{\phi(t)}(\phi(t)-\rho)^{n-1}(3R+|\rho|)^{(n-3)p'/2-p'q}d\rho\\
&=& C\langle \phi(t)\rangle^{-(n-1)p'/2}\{I_1(t)+I_2(t)\},
\end{eqnarray*}
where
$$I_1(t)=\int_{-R}^{\phi(t)/2} (\phi(t)-\rho)^{n-1} (3R+|\rho|)^{(n-3)p'/2-p'q} d\rho$$
and
$$I_2(t)=\int_{\phi(t)/2}^{\phi(t)} (\phi(t)-\rho)^{n-1} (3R+\rho)^{(n-3)p'/2-p'q} d\rho.$$
Since $(n-3)p'/2-p'q+1<0$, integration by parts yields that
\begin{eqnarray*}
I_2(t)&\le& C\langle \phi(t)\rangle^{n-1} (3R+\phi(t))^{(n-3)p'/2-p'q+1} \\
&&\qquad-C\int_{\phi(t)/2}^{\phi(t)} (\phi(t)-\rho)^{n-2} (3R+\rho)^{(n-3)p'/2-p'q+1}d\rho\\
&\le& C\langle \phi(t)\rangle^{n-1+(n-3)p'/2-p'q+1}.
\end{eqnarray*}
Similarly, we have
\begin{eqnarray*}
I_1(t)&\le& C(\phi(t)+R)^{n-1}-
\int_{-R}^{\phi(t)/2}(\phi(t)-\rho)^{n-2}(3R+\rho)^{(n-3)p'/2-p'q+1}d\rho\\
&\le& C(\phi(t)+R)^{n-1}\le C\langle \phi(t)\rangle^{n-1}.
\end{eqnarray*}
Therefore, we obtain \eqref{upp_I(t)}.
\end{proof}

\section{Iteration argument}
In this section, we apply iteration argument to give the proof of Theorem \ref{Thm:sub-critical} and Theorem \ref{Thm:critical}.
\subsection{Proof of Theorem \ref{Thm:sub-critical}}
In this subsection, we devote to the sub-critical case. Firstly, we define the functional:
$$G(t)=\int_{\mathbb{R}^n} u dx.$$
As $u$ has compact support in $B(0,\phi(t)+R)$, by H\"{o}lder inequality we have:
$$\int_{\mathbb{R}^n}|u|^pdx\geqslant (R+\phi(t))^{-n(p-1)}|G(t)|^p\ \ \ \ \ \mbox{for} \ \ t\geqslant0.$$
Moreover, choosing the test function $\zeta=\zeta(x,s)$ in \eqref{weak} to satisfy $\zeta\equiv1$ in $\{(x,s)\in\mathbb{R}^n\times[0,t]:|x|\leqslant \phi(s)+R\}$, the integration by parts gives
$$G''(t)=\int_{\mathbb{R}^n}|u|^pdx.$$
Combing the above and integrate twice, we have iteration frame inequality.
\begin{equation}\label{iter2}G(t)\geqslant \int_0^td\tau\int_0^\tau(R+\phi(s))^{-n(p-1)}|G(s)|^pds.\end{equation}
Second, in order to initiate the iteration argument, we need following lower bound estimate, by plugging the lower bound estimate \eqref{lp_lower} and integrate twice,
\begin{eqnarray*}
G(t)&\geqslant&\int_0^td\tau\int_0^\tau C_1\varepsilon^p(1+s)^{\f p2}(1+\phi(s))^{n-1-\frac{n}{2}p}ds\\
&\geqslant&C_1\varepsilon^p\int_{T_0}^t\int_{T_0}^\tau(1+s)^{\f p2+\f {m+2}2(n-1-\f n2 p)}ds\\
&\geqslant&C_1\varepsilon^p\int_{T_0}^t(1+\tau)^{-\mu-(n+\mu-1)\frac p2}d\tau\int_{T_0}^\tau(1+s)^{\f {m+2}2(n-1)+\mu}ds\\
&\geqslant&C_1\varepsilon^p(1+t)^{-\mu-(n+\mu-1)\frac p2}\int_{T_0}^td\tau\int_{T_0}^\tau(s-T_0)^{\f {m+2}2(n-1)+\mu}ds,
\end{eqnarray*}
where $\mu=\f {mn}2$.
That implies
\begin{equation}\label{j=1}
G(t)\geqslant C_2\varepsilon^p(1+t)^{-\mu-(n+\mu-1)\frac p2}(t-T_0)^{\f {m+2}2(n-1)+\mu+2}\ \ \mbox{for}\ \ t>T_0
\end{equation}
where $C_2=\frac{C_1}{(\f {m+2}2(n-1)+\mu+1)(\f {m+2}2(n-1)+\mu+2)}$.

Now we begin our iteration argument. Assume that
\begin{equation}\label{iter-assu}
G(t)>D_j(1+t)^{-a_j}(t-T_0)^{b_j}\ \ \ \mbox{for}\ \ t>T_0,\ \ j=1,2,3\ \cdots
\end{equation}
with positive constants $D_j,\ a_j$ and $b_j$ determined later. \eqref{j=1} asserts \eqref{iter-assu} is true for $j=1$ with
\begin{equation}\label{series1}D_1=C_2\varepsilon^p,\ \ a_1=\mu+(n+\mu-1)\frac p2,\ \ \ b_1=\f {m+2}2(n-1)+\mu+2.\end{equation}
Plugging \eqref{iter-assu} into \eqref{iter2}, we have for $t>T_0$
\begin{eqnarray*}
G(t)&>&C_0\int_0^td\tau\int_0^\tau(M+s)^{-\f {m+2}2n(1-p)}|G(s)|^pds\\
&>&C_0\int_{T_0}^td\tau\int_{T_0}^\tau(1+s)^{-\f {m+2}2n(1-p)}D_j^p(1+s)^{-pa_j}(s-T_0)^{pb_j}ds\\
&>&C_0D_j^p(1+t)^{-\mu-\f {m+2}2 n(1-p)-pa_j}\int_{T_0}^t\int_{T_0}^\tau(s-T_0)^{\mu+pb_j}dsd\tau\\
&>&\frac{C_0D_j^p}{(\mu+pb_j+1)(\mu+pb_j+2)}(1+t)^{-\mu-\f {m+2}2 n(1-p)-pa_j}(t-T_0)^{\mu+pb_j+2}.
\end{eqnarray*}
So the assumption \eqref{iter-assu} is true if the sequence $\{D_j\}$, $\{a_j\}$, $\{b_j\}$ are define by
\begin{equation}\label{series2}D_{j+1}\geqslant\frac{C_0}{(\mu+pb_j+2)^2}D_j^p,\ \ a_{j+1}=\mu+\f {m+2}2 n(p-1)+pa_j,\ \ b_{j+1}=\mu+2+pb_j.\end{equation}
It follows from \eqref{series1} and \eqref{series2} that for $j=1,2,3\cdots$
\begin{eqnarray}
a_j&=&[\mu+(n+\mu-1)\frac p2+\f {m+2}2 n+\frac{\mu}{p-1}]p^{j-1}-(\f {m+2}2 n+\frac{\mu}{p-1})\nonumber\\
&=&\alpha p^{j-1}-(\f {m+2}2 n+\frac{\mu}{p-1})\label{a_j}\\
b_j&=&[\f {m+2}2(n-1)+\mu+2+\frac{\mu+2}{p-1}]p^{j-1}-\frac{\mu+2}{p-1}\nonumber\\
&=&\beta p^{j-1}-\frac{\mu+2}{p-1}\label{b_j}
\end{eqnarray}
where we denote the positive constants $$\alpha=\mu+(n+\mu-1)\frac p2+\f {m+2}2 n+\frac{\mu}{p-1},$$ $$\beta=\f {m+2}2(n-1)+\mu+2+\frac{\mu+2}{p-1}.$$
We employ the inequality
$$b_{j+1}=pb_j+\mu+2<p^j[\f {m+2}2(n-1)+\mu+2+\frac{\mu+2}{p-1}],$$
for $D_{j+1}$ to obtain
$$D_{j+1}\geqslant C_3\frac{D^p_j}{p^{2j}}$$
where $$C_3=\frac{C_0}{(\f {m+2}2(n-1)+\mu+2+\frac{\mu+2}{p-1})^2}.$$
Hence,
\begin{eqnarray*}
\log D_j&\geqslant& p\log D_{j-1}-2(j-1)\log p+\log C_3\\
&\geqslant& p^2\log D_{j-2}-2(p(j-2)+(j-1))\log p+(p+1)\log C_3\\
&\geqslant&\cdots\\
&\geqslant&p^{j-1}\log D_1-2\log p\sum_{k=1}^{j-1}kp^{j-1-k}+\log C_3\sum_{k=1}^{j-1}p^k.
\end{eqnarray*}
Direct calculation gives
$$\sum_{k=1}^{j-1}kp^{j-1-k}=\frac{1}{p-1}(\frac{p^j-1}{p-1}-j)$$ and
$$\sum_{k=1}^{j-1}p^k=\frac{p-p^j}{1-p},$$
which yields
\begin{eqnarray*}
\log D_j&\geqslant& p^{j-1}\log D_1-\frac{2\log p}{p-1}(\frac{p^j-1}{p-1}-j)+\log C_3\frac{p-p^j}{1-p}\\
&=&p^{j-1}\bigg(\log D_1-\frac{2p\log p}{(p-1)^2}+\frac{p\log C_3}{p-1}\bigg)+\frac{2\log p}{p-1}j+\frac{2\log p}{(p-1)^2}+\frac{p\log C_3}{1-p}
\end{eqnarray*}
Consequently for $j>\left[\frac{p\log C_3}{2\log p}-\frac{1}{p-1}\right]+1$,
\begin{equation}\label{D_j}D_j\geqslant\exp\{p^{j-1}(\log D_1-S_p(\infty))\}\end{equation}
with
$$ S_p(\infty):=\frac{2p\log p}{(p-1)^2}-\frac{p\log C_3}{p-1}.$$
Inserting \eqref{a_j}, \eqref{b_j} and \eqref{D_j} into \eqref{iter-assu} gives
\begin{eqnarray}
G(t)&\geqslant&\exp\big(p^{j-1}(\log D_1-S_p(\infty))\big)(1+t)^{-\alpha p^{j-1}+(n+\frac{\mu}{p-1})}(t-T_0)^{\beta p^{j-1}-\frac{\mu+2}{p-1}}\nonumber\\
&\geqslant&\exp\big(p^{j-1}J(t)\big)(1+t)^{n+\frac\mu{p-1}}(t-T_0)^{-\frac{\mu+2}{p-1}}\label{contr}
\end{eqnarray}
where
$$J(t):=\log D_1-S_p(\infty)-\alpha\log(1+t)+\beta\log(t-T_0).$$
For $t>2T_0+1$, we have
\begin{eqnarray*}
J(t)&\geqslant& \log D_1-S_p(\infty)-\alpha\log(2t-2T_0)+\beta\log(t-T_0)\\
&\geqslant&\log D_1-S_p(\infty)+(\beta-\alpha)\log(t-T_0)-\alpha\log2\\
&=&\log(D_1\cdot(t-T_0)^{\beta-\alpha})-S_p(\infty)-\alpha\log2.
\end{eqnarray*}
Note that
\begin{eqnarray*}
\beta-\alpha&=&\left [\f {m+2}2(n-1)+\mu+2+\frac{\mu+2}{p-1}\right]-\left[\mu+(n+\mu-1)\frac p2+\f {m+2}2 n+\frac{\mu}{p-1}\right]\\
&=& \f {-((m+2)\f n2-1)p^2-((m+2)(1-\f n2)-3)p+(m+2)}{2(p-1)}\\
&>&0 \quad\quad\quad \mbox{for}\quad 1<p<p_{crit}(m,n).
\end{eqnarray*}
Thus if $$t>\max\{T_0+(\frac{e^{[S_p(\infty)+\alpha\log2]+1}}{C_2\varepsilon^p})^{2(p-1)/\gamma(m,n,p)},2T_0+1\},$$
we then get $J(t)>1$, and this in turn gives $G(t)\rightarrow\infty$ by taking $j\rightarrow\infty$ in \eqref{contr}. Therefore, for $\varepsilon<\varepsilon_0$, we obtain the desired upper bound,
$$T\leqslant C_4\varepsilon^{-\frac{2p(p-1)}{\gamma(m,n,p)}}$$
with
$$C_4:=\left(\frac{e^{(S_p(\infty)+\alpha\log2)+1}}{C_2}\right)^{2(p-1)/\gamma(m,n,p)}.$$
This completes our proof of Theorem \ref{Thm:sub-critical}.

\subsection{Proof of Theorem \ref{Thm:critical}}
In this subsection, we focus on the proof of critical case. The following lemma initiates the iteration.
\begin{lem}
\label{lem:initiate}
Suppose that the assumptions in Theorem \ref{Thm:critical} are fulfilled.
Then, $\d F(t)=\int_{\R^n} u(x,t) \eta_{q}(x,t,t)dx$ for $t\ge 3/2$ satisfies that
\begin{equation}
\label{initiate}
\langle t\rangle^{\f m4}F(t)\ge M\e^p \log \{t/(3/2)\},
\end{equation}
where $M=C_0B_1/3^3$ and $C_0$ is the one in Lemma \ref{lem:lp_lower}.
\end{lem}

\begin{proof}
Putting the estimates \eqref{lp_lower} and (ii) with $q=(n-1)/2-1/p>0$ in Lemma \ref{lem2} into \eqref{identity}, we
get
$$F(t)\ge\int_0^t(t-s)\int_{}|u(s)|^p\eta_q(x,t,s)dxds \ge \frac{C_0B_1\e^p}{\langle t\rangle^{1+\f m4}}\int_{0}^{t}\frac{\langle s\rangle^{\f p2}(t-s)}
{\langle \phi(s)\rangle^{q+pn/2-(n-1)+1-\f{m+4}{2(m+2)}}}ds.$$
Since $\gamma(m,n,p)=0$ implies that
\begin{eqnarray*}
&&-\f p2+[q+\frac{np}{2}-(n-1)+1-\f{m+4}{2(m+2)}]\f{m+2}{2}\\
&=&-\f1{2p}\gamma(m,n,p)+1=1,
\end{eqnarray*}
the above inequality may simplify as
$$\langle t\rangle^{\f m4}F(t)\ge\frac{C_0B_1\e^p}{3^2 t}\int_{1}^{t}\frac{t-s}{s}ds.$$
Then for $t\ge 3/2$, the integration by parts gives
$$\langle t\rangle^{\f m4}F(t)\ge \frac{C_0B_1\e^p}{3^2 t}\int_{2t/3}^{t}\log s ds\ge\frac{C_0B_1\e^p}{3^3} \log (2t/3).$$
The proof is complete.
\end{proof}

As the iteration frame of the logarithmic type has been setup in Proposition \ref{prop:frame}, combing the initial estimate \eqref{initiate} in Lemma \ref{lem:initiate}, we may establish following Proposition by induction on $j$.
\begin{prop}
\label{prop:j-step}
Suppose that the assumptions in Theorem \ref{Thm:critical} are fulfilled.
Then, $\d F(t)=\int_{\R^n} u(x,t) \eta_{q}(x,t,t)dx$ for $t\ge l_j $ $(j\in \N)$ satisfies that
\begin{equation}
\label{j-step}
\langle t\rangle^{\f m4}F(t)\ge C_j (\log \langle t\rangle)^{-b_j}\left\{ \log \left(t/l_j \right)\right\}^{a_j},
\end{equation}
where $\d l_j=l_0+\sum_{k=1}^{j}2^{-(k+1)}=2-2^{-(j+1)}$ $(j\in\N)$ with $l_0=3/2$.
Here, $a_j$, $b_j$ and $C_j$ are defined by
\begin{equation}
\label{a_j,b_j}
a_j=\frac{p^{j+1}-1}{p-1}\quad \mbox{and}\quad b_j=p^{j}-1,
\end{equation}
\begin{equation}
\label{ind_C_j}
C_j=\exp\{p^{j-1}(\log(C_1(2p)^{-S_j}E^{1/(p-1)})-\log E^{1/(p-1)})\}\quad (j\ge 2),
\end{equation}
\begin{equation}
\label{C_1}
C_1=N\e^{p^2},
\end{equation}
where $C$ is the one in (\ref{frame}) and
\begin{equation}
\label{S_j,E}
N=\frac{C M^p}{3^27(p+1)},\quad S_j=\sum_{i=1}^{j-1}\frac{i}{p^i},\quad E=\frac{C(p-1)}{2^33^2p^2}.
\end{equation}
\end{prop}
We omit the proof details, as it is exactly same as Proposition 5.3 in \cite{WY1801}, once we replace $F(t)$ by $\langle t\rangle^{\f m4}F(t)$.
Theorem \ref{Thm:critical} then can be finalized with same argument in \cite{WY1801}.

{\bf Acknowledgment}:
The first author is supported by NSFC No. 11501511 and Zhejiang Provincial Nature Science Foundation of China under Grant No. LQ15A010012.
The second author is supported by Zhejiang Provincial Nature Science Foundation of China under Grant No. LY18A010023.
\newpage
\bibliographystyle{elsarticle-num}
\bibliography{References}

\end{document}